\documentclass[12pt]{amsart}

\usepackage{amsmath, color} \usepackage{amssymb} \usepackage{amsthm}
\usepackage{amscd}
\usepackage{tikz}
\usetikzlibrary{calc}
\usetikzlibrary{patterns}
\usetikzlibrary{matrix}
\usepgflibrary{arrows}
\usetikzlibrary{arrows}
\usetikzlibrary{decorations.pathreplacing}
\usetikzlibrary{decorations.pathmorphing}
\usepackage{epsfig}
\usepackage{enumerate}

\def\Tor{\operatorname{Tor}}

\def\p{{\mathrm{pol\;}}} 

\DeclareMathOperator{\st}{st}
\DeclareMathOperator{\reg}{reg}

\DeclareMathOperator{\mingen}{Mingens}

\theoremstyle{plain}
\newtheorem{theorem}{Theorem}[section]
\newtheorem{corollary}[theorem]{Corollary}
\newtheorem{proposition}[theorem]{Proposition}

\newtheorem{definition}[theorem]{Definition}
\newtheorem{discussion}[theorem]{Discussion}
\newtheorem{lemma}[theorem]{Lemma}
\newtheorem{question}[theorem]{Question}
\newtheorem{eg}[theorem]{Example}

\newtheorem{observation}[theorem]{Observation}

\begin{document} 

\title{The Regularity of Powers of Edge Ideals}
\date{\today}

\author[Banerjee]{Arindam Banerjee}
\address{Department of Mathematics, University of Virginia,
Charlottesville, VA, USA} \email{ab4cb@virginia.edu}

\subjclass[2000]{Primary 13-02, 13F55, 05C10}

\baselineskip 16pt \footskip = 32pt

\begin{abstract}
In this paper we prove the existence of a special order on the set of minimal monomial generators of powers of edge ideals of arbitrary graphs. Using this order we find new upper bounds on the regularity of powers of edge ideals of graphs whose complement does not have any induced four cycle. 
\end{abstract}

\maketitle \markboth{A.Banerjee}
{Bounds on regularity}

\section{Introduction}
\bigskip

In this work we find new upper bounds for the regularity of some classes of monomial ideals associated to graphs.   Our original motivation is the following question, which is the base case of the Open Problem $1.11(2)$ in [13]:

\begin{question}\label{ques1}
Let $I(G)$ be the edge ideal of a graph $G$ which does not have any induced four cycle in its complement. If $\reg(I(G)) \leq 3$, then is it true that for all $s\geq 2$, $I(G) ^ s$ has linear minimal free resolution? 

\end{question}

Bounds on the regularity of edge ideals have been studied  by a number of researchers (see [1], [2], [3], [4], [5], [6], [7],[8] [9], [11],[12],[13]). For example, Fr\"oberg (see [3]) has shown that, when $I(G)$ is the edge ideal of a graph whose complement does not have any induced cycle of size greater than or equal to four, then $I(G)$ has linear minimal free resolution.

  We are interested in finding upper bounds on the regularities of the higher powers of $I(G)$. Herzog, Hibi and Zheng have shown in [6] that if $I(G)$ is the edge ideal of a graph $G$ which has no induced cycle of length greater than or equal to four in its complement (that is $I(G)$ has linear minimal free resolution) then for all $s\geq 2$, $I(G)^s$ has linear minimal free resolution. Fransisco, H\`a and Van-Tuyl have further shown that if $I(G)^s$ has linear minimal free resolution for some $s$, then $G$ has no induced four cycle in its complement ( Proposition $1.8$ in [13]). These two results lead us to study bounds on the regularity of powers of $I(G)$ when $G$ has no induced four cycle in its complement. Our main result is Theorem $6.17$ where we prove all higher powers of edge ideals of a gap free (equivalently, no induced four cycle in complement, as observed in section $2$) and cricket free (defined in section $2$) graph have linear minimal free resolution. More precisely: \\\\
  
\begin{theorem}
For any gap free and cricket free graph $G$ and for all $s \geq 2$, $\reg (I(G)^s) = 2s$ and as a consequence $I(G)^s$ has a linear minimal free resolution.\\\\
\end{theorem}
  
    This partilally answers Question $1.1$, as we prove in section $3$ that edge ideals of gap free and cricket free graphs have regularity less than or equal to $3$ (Theorem $3.4$). As claw free graphs (defined in section $2$) are automatically cricket free, our results generalize a previous result by E. Nevo (Theorem $1.2$ of [12]) that says the edge ideals of gap free and claw free  graphs have regularity less than or equal to $3$ and their squares have linear minimal free resolutions.\\\\
    
   In order to prove Theorem $6.17$, we first show that the minimal monomial generators of powers of edge ideal $I(G)$ for any finite simple graph $G$ have specific order that satisfies some nice property (Lemma $4.11$, Theorem $4.12$). More precisely:\\\\
\begin{theorem}
For each $n\geq 1$ there exists an ordered list $L^{(n)}$ of minimal monomial generators of $I(G)^n$ which satisfies the following property:\\For all $k\geq 1$ and for all $j\leq k$, if $(L_j ^{(n)}: L_{k+1} ^{(n)} )$  is not contained in $(I(G)^{n+1}:L_{k+1} ^ {(n)})$ then  there exists $i \leq k$, such that $(L_i ^{(n)}: L_{k+1} ^{(n)})$ is generated by a variable and $(L_j ^{(n)} : L_{k+1} ^{(n)}) \subseteq (L_i ^{(n)} : L_{k+1} ^ {(n)} )$. For monamials $m$ and $n$, $(m:n)$ stands for $((m):(n))$.\\\\
\end{theorem}
   Using this ordering we shall prove that $\reg (I(G)^n)$ is bounded above by the maximum of $\reg (I(G)^n :e_1...e_{n-1})+2n-2$ for all possible $(n-1)$-fold products of edges $e_1...e_{n-1}$ and $\reg (I(G)^{n-1})$ (See Theorem $5.2$). Next we prove that the ideals $(I(G)^n :e_1...e_{n-1})$ are quadratic monomial ideals with generators satisfying certain conditions (See Theorems $6.1,6.5,6.7$). Finally, by using polarization technique we get edge ideals corresponding to these quadratic monomial ideals with same regularity (See [9], Section $3.2$ and Exercise $3.15$ of [10] for details) and using Fr\"oberg's theorem (See Theorem $1$ of [3] and Theorem $[1.1]$ of [13]) get bounds on them. As a consequence we also get a different proof of the Herzog, Hibi and Zheng's result mentioned above (Theorem $6.16$).\\\\

\bigskip

\section{Preliminaries}


Throughout this paper, we let $G$ be a finite simple graph with vertex set $V(G)$. 
For $u, v \in V(G)$, we let $d(u,v)$ denote the \emph{distance} between $u$ and $v$, the
fewest number of edges that must be traversed to travel from $u$ to $v$.    

A subgraph $G' \subseteq G$ is called \emph{induced} if $uv$ is an edge of $G'$ whenever $u$ and $v$ are vertices of $G'$ and $uv$ is an edge of $G$.

The \emph{complement} of a graph $G$, for which we write $G^c$, is the graph on the same vertex set in which $uv$ is an edge of $G^c$ if and only if it is not an edge of $G$.  

Finally, let $C_k$ denote the cycle on $k$ vertices, and we let $K_{m, n}$ denote the complete bipartite graph with $m$ vertices on one side, and $n$ on the other.  

\begin{definition}
Let $G$ be a graph.  We say two disjoint edges $uv$ and $xy$ form a \emph{gap} in $G$ if $G$ does not have an edge with one 
endpoint in $\{u,v\}$ and the other in $\{x,y\}$.
A graph without gaps is called \emph{gap-free}.  Equivalently, $G$ is gap-free if and only if $G^c$ contains no induced $C_4$.
\end{definition}

Thus, $G$ is gap-free if and only if it does not contain two vertex-disjoint edges as an induced subgraph.  


\begin{definition}
Any graph isomorphic to $K_{1, 3}$ is called a \emph{claw}. Any graph isomorphic to $K_{1,n}$ is called an \emph{n-claw}. If $n >1$, the vertex with degree $n$ is called the root in $K_{1,n}$.
A graph without an induced claw is called \emph{claw-free}. A graph without an induced \emph{n-claw} is called \emph{n-claw-free}.
\end{definition}

\begin{definition}
Any graph isomorphic to the graph with set of vertices $\{w_1,w_2,w_3,w_4,\\w_5\}$ and set of edges $\{w_1w_3,w_2w_3,w_3w_4,w_3w_5,w_4w_5\}$ is called a cricket.
A graph without an induced cricket is called \emph{cricket-free}. 
\end{definition}

\begin{definition}
An edge in a graph is called a whisker if any of its vertices has degree one.
\end{definition}

\begin{definition}
A graph is called anticycle if its complement is a cycle.
\end{definition}

\begin{observation}
A claw-free graph is cricket-free.
\end{observation}


 If $G$ is a graph without isolated vertices then let $S$ denote the polynomial ring on the vertices of $G$ over some fixed field $K$.  Recall that the \emph{edge ideal} of $G$ is 
\[
I(G) = (xy: xy \text{ is an edge of } G).
\]


\begin{definition}
Let $S$ be a standard graded polynomial ring over a field $K$. The Castelnuovo-Mumford regularity of a finitely generated graded $S$ module $M$, written $\reg(M)$ is given by $$\reg(M):= \max \{j-i|\Tor_{i} (M,K)_j \neq 0 \}$$
\end{definition}

\begin{definition}
We say that $I(G)^s$ is \emph{$k$-steps linear} whenever the minimal free resolution of $I(G)^s$ over the polynomial ring
is linear for $k$ steps, i.e., $\Tor_{i}^S(I(G)^s,K)_j = 0$ for all $1\leq i\leq k$ and all $j\ne i+2s$. We say $I(G)$ has linear minimal free resolution if the minimal free resolution is $k$-steps linear for all $k \geq 1$.
\end{definition}

 We end this section by recalling a few well known results. We refer reader to [1] and [13] for reference.

\begin{observation} 
Let $I(G)$ be the edge ideal of a graph $G$. Then $I(G)^s$ has linear minimal free resolution if and only if $\reg (I(G)^s)=2s$.
\end{observation}

\begin{lemma}
Let $I \subseteq S$ be a monomial ideal. Then for any variable $x$, $\reg(I,x)\leq \reg(I)$. In particular if $v$ is a vertex in a graph $G$, then $\reg (I(G-v))\leq \reg((I(G))$.
\end{lemma}

 The following theorem follows from Lemma $2.10$ of [1]:

\begin{lemma}\label{exact}
Let $I \subseteq S$ be a monomial ideal, and let $m$ be a monomial of degree $d$.  Then
\[
\reg(I) \leq \max\{ \reg (I : m) + d, \reg (I,m)\}. 
\]
Moreover, if $m$ is a variable $x$ appearing in I, then $\reg(I)$ is {\it equal} to one of
these terms.
\end{lemma}

 Finally the following theorem due to Fr\"oberg (See Theorem $1$ of [3] and Theorem $1.1$ of [13]) is used repeatedly throughout this paper:

\begin{theorem}
The minimal free resolution of $I(G)$ is linear if and only if the complement graph $G^c$ is chordal, that is no induced cycle in $G^c$ has length greater than three.
\end{theorem}

\bigskip

\section{Gap-free graphs}

In this section we observe some basic results concerning gap-free graphs and their regularity. We prove that a cricket free and gap free graph has regularity at most $3$, generalizing Nevo's result (Theorem $3.3$ of [1]) that a gap free and claw free graph has regularity at most $3$. We generalize Nevo's result in another direction by proving an n-claw free and gap free graph has regularity at most $n$.
\begin{definition}
For any graph $G$, we write $\reg(G)$ as shorthand for $\reg(I(G))$.  
\end{definition}

Recall that the \emph{star} of a vertex $x$ of $G$, for which we write $\st x$, is given by
\[
\st x = \{y \in V(G) : xy \text{ is an edge of }G\} \cup \{x\}.
\]
The following lemma is Lemma $3.1$ of [1], which we shall use a lot in this work.

\begin{lemma}\label{removevertex}
Let $x$ be a vertex of $G$ with neighbors $y_1,y_2,..,y_m$.  Then
\[
(I(G) : x) = (I(G - \st x),y_1,..,y_m) \text{ and } (I(G), x) = (I(G - x),x).
\]
Thus, $\reg(G) \leq \max \{ \reg(G - \st x ) + 1, \reg(G - x)\}$. Moreover, $\reg(G)$ is equal to one of
these terms.

\end{lemma}

 The next proposition is Proposition $3.2$ of [1]. 

\begin{proposition}\label{distance2}
Let $G$ be gap-free, and let $x$ be a vertex of $G$ of highest degree.  Then $d(x, y) \leq 2$ for all vertices $y$ of $G$.  
\end{proposition}

 We prove the next two theorems using Proposition $3.3$. Our proof is motivated by the proof of Theorem $3.3$ of [1].

\begin{theorem}
Suppose $G$ is both cricket-free and gap-free.  Then $\reg(G) \leq 3$.  
\end{theorem}

\begin{proof}
Let $x$ be a vertex of maximum degree. As $G$ is gap free and cricket free, so is $G-x$. By induction, $G-x$ has regularity less than or equal to $3$. Because of Lemma $3.2$ and Theorem $2.12$, it is enough to show that $(G-\text{st }x)^c$ has no induced cycle of length greater than or equal to $4$. As $G$ is gap free, so is $(G-\text{st }x)$; hence, $(G-\text{st }x)^c$ has no induced $4-$cycle. So it is enough to show it does not have an induced cycle of length greater than or equal to $5$.\\

  Let $\{y_1,y_2,y_3,y_4,...,y_n\}$ be an induced cycle ($n\geq 5$) in $(G-\text{st }x)^c$; because of Proposition 3.3, there is a $w$ such that $xw$ and $wy_1$ are edges in $G$. As $y_2 y_n$ is an edge in $G$, and neither $y_1 y_2$ nor $y_1 y_n$ are edges in $G$, either $wy_2$, $wy_n$ or both are edges in $G$. If both are edges then $\{x,w,y_1,y_2,y_n\}$ forms an induced cricket.\\
  
  Suppose only one of them is an edge. Without loss of generality, we may assume $wy_2$ is an edge. As $y_3y_n $ is an edge in $G$, and G gap free, $wy_3$ is an edge in $G$; otherwise $\{x,w,y_3,y_n \}$ forms a gap in $G$. This makes  $\{x,w,y_1,y_2,y_3 \}$  an induced cricket.  
\end{proof}

\begin{theorem}
The edge ideal of a graph which is gap free and $n$-claw free, has regularity less than or equal to $n$.
\end{theorem}

\begin{proof}
For $n=3$, this was proved by E. Nevo and this is Theorem $3.3$ of [1]. So we may assume $n \geq 4$. Let $x$ be a vertex with maximum degree. Because of Lemma $3.2$, it is enough to show $G-\text{st } {x}$ has regularity less than or equal to $n-1$; as $G-{x}$ has regularity less than or equal to $n$ by induction on number of vertices. Hence, it is enough to show $G-\text{st } {x}$ is $(n-1)$-claw free.\\ 

   If $a_1,a_2,a_3,...,a_n$ is a $(n-1)$-claw with root $a_1$ in $G-\text{st } {x}$ then any $w$ in the neighborhood of $x$ is either connected to $a_1$ or all of $a_2,a_3,..,a_n$; otherwise if $w$ is not connected to $a_1$ and $a_i$ then $xw$ and $a_1 a_i$ will form a gap. If $a_1$ is connected to all neighbors of $x$, it has a degree strictly more than $x$, which is contradictory to the assumption that $x$ is a vertex with maximum degree. Hence, there, is a neighbor $w$ which is not connected to $a_1$ but is connected to all of $a_2,a_3,..,a_n$. As $x$ is not connected to any of the $a_i$s, $\{x,w,a_2,a_3,..,a_n\}$ forms an $n$-claw with root $w$, which is  contradictory to the hypothesis.
\end{proof}

\section{Ordering the minimal monomial generators of powers of edge ideals}
\medskip

\begin{discussion} Let the set of minimal monomial generators of any ideal $J \subset S$ be denoted by $\mingen(J)$. Let $I$ be an arbitrary edge ideal. Set $\mingen(I)=\{L_1, L_2,....,L_k\}$. We give $\mingen(I)$ the follwing order: $L_1>L_2>...>L_k$. We will put an order on $\mingen(I^n)$ for all integers $n\geq 2$ as follows: For $n >1$, we say $M>N$ for  $M,N \in \mingen(I^n)$ if there exists an expression $L_1 ^{a_1} L_2 ^{a_2}...L_k ^{a_k} =M$ such that for all expressions $L_1 ^{b_1}...L_k ^{b_k}=N$, we have $(a_1,...,a_k)>_{\text{lex}} (b_1,...,b_k)$. If $(a_1,...,a_k) \geq_{\text{lex}} (c_1,...,c_k)$ for all $(c_1,....,c_k)$ such that $L_1^{c_1}....L_k^{c_k}=M$ then $L_1 ^{a_1} L_2 ^{a_2}...L_k ^{a_k}$ is called a maximal expression of $M$. Let $L^{(n)}$ be the totally ordered set of minimal monomial generators of $I^n$, ordered in the way discussed above.\\\\
\end{discussion}

\begin{definition}
 If $m_1$ is a minimal monomial generator of $I^k$ and $m_2$ is a minimal monomial generator of $I^n$ where $n > k$, we say $m_1$ divides $m_2$ as an edge and use the notation $m_1 |^{\text{edge}} m_2$, if there exists $m_3$, a minimal monomial generator of $I^{n-k}$ with $m_2=m_1m_3$.\\
\end{definition}

\begin{eg}
If $I=(ab,bc,ad,bd)$ then $ab |^{\text{edge}} ab^2d$ as $bd=\frac{ab^2d}{ab}$ is a minimal monomial generator of $I$ but $ab\nmid ^{\text{edge}} abcd$ as $cd=\frac{abcd}{ab}$ is not a minimal monomial generator of $I$.
\end{eg}

\begin{discussion} We have the following for the list $L^{(n)}$ created above:\\\\
1. $L^{(1)}=L:=\{L_1>....>L_k\}$\\\\
2. For any minimal monomial generator $m$ of $I^n$, $n \geq 2$, the maximal expression of $m$, is an expression of $m$ as a product of $n$ elements of $L$, $m=L_{i_1} L_{i_2}...L_{i_n}$, where:\\ a. $i_1$ is the minimum integer such that $L_{i_1} |^{\text{edge}} m$ \\ b. For all $l \geq 1$, $i_{l+1}$ is the minimal integer such that $L_{i_{l+1}} |^{\text{edge}} \frac{m}{L_{i_1}...L_{i_l}}$. For any edge $cd$ we say $cd$ is a part of the maximal expression of $m$ if $cd=L_{i_k}$ for some $k$. \\
This expression is unique by the construction.\\\\
3. For two minimal monomial generators $m_1, m_2$ with maximal expressions $m_1=L_{i_1}...L_{i_n}$ and $m_2=L_{j_1}...L_{j_n}$, we have $m_1 >_{\text{lex}} m_2$ if for the minimum integer $l$ such that $i_l \neq j_l$, $i_l < j_l$. \\\\
4. If $L_i$ and $L_j$ are two generators of $I$ with $i<j$, then we say $``L_j$ comes after $L_i$'' or $``L_i$ comes before $L_j$''. \\
\end{discussion}

\begin{eg}
Let $I=(ab,bc,ad,bd)$. Let $L^{(1)}=\{ab>bc>ad>bd\}$. Then $L^{(2)}=\{a^2 b^2>ab^2 c>a^2bd>ab^2d>b^2c^2>abcd>b^2cd>
a^2d^2>abd^2>b^2d^2\}$.\\
\end{eg}

\begin{definition}
 If $L_i=ab$ is an edge, that is a minimal monomial generator of $I$,  and $m$ is a minimal monomial generator of $I^n$, $n \geq 2$, then we say $m$ belongs to $ab$, or $m$ belongs to $L_i$, if $i$ is the least integer such that $L_i |^{\text{edge}} m$.\\\\
\end{definition}

\begin{eg}
Let $I=(ab,bc,ad,bd)$ with $L=L^{(1)}=\{ab>bc>ad>bd\}$. Then $abcd$ belongs to $L_2=bc$ as $ab \nmid ^{\text{edge}} abcd$ and $bc | ^{\text{edge}} abcd$ and $ab^2d$ belongs to $L_1=ab$ as $ab |^{\text{edge}} ab^2d$.\\

\end{eg}

 We record several easy observations that we need in the sequel.

\begin{observation} For two minimal monomial generators $m_1, m_2$, if $m_1$ belongs to an edge $L_i$ and $m_2$ belongs to another edge $L_j$ with $i < j$, then $m_1 >_{\text{lex}} m_2$.\\\\
\end{observation}

\begin{observation}
 For two minimal monomial generators $m_1, m_2$ of $I^n$ which both belong to an edge $L_i$, we see that $m_1 >_{\text{lex}} m_2$ if and only if $\frac{m_1}{L_i} >_{\text{lex}} \frac{m_2}{L_i}$.\\\\
\end{observation}

\begin{observation}
Suppose $m$ is a minimal monomial generator of $I^n$, $n\geq 2$, and $gh$ is an edge which is a part of the maximal expression of $m$. Write $m=gh m'$. For any minimal monomial generator $m''$ of $I^{n-1}$ such that $m'' >_{\text{lex}} m'$, then $gh m'' >_{\text{lex}} m$.\\\\
\end{observation}

\begin{proof}
 Let $L=\{L_1>L_2>....>L_k\}$. Let $gh=L_j$ for some $j$. Let $m''=L_1^{a_1} L_2^{a_2}....L_k^{a_k}$ be the maximal expression of $m''$ and $m'=L_1^{b_1}L_2^{b_2}....L_k^{b_k}$ be the maximal expression of $m'$. As $gh$ is part of the maximal expression of $m$, the maximal expression of $m$ is $L_1^{b_1}....L_j^{b_j +1}....L_k^{b_k}$. As by assumption $(a_1,...,a_j,...a_k) >_{\text{lex}} (b_1,....,b_j,...,b_k)$, we have $(a_1,...,a_j +1,....a_k) >_{\text{lex}} (b_1,....,b_j +1,...b_k)$. Now\\ $L_1^{a_1}....L_j^{a_j+1}....L_k^{a_k}$ is an expression for $ghm''$. Hence $ghm'' >_{\text{lex}} ghm'=m$. 
\end{proof}

 The next lemma is the most important technical result of this paper as it allows us to build the framework of Section $5$. Using the framework of Section $5$ we obtain our bounds in Section $6$.\\\\

\begin{lemma}
For all $k\geq 1$ and for all $j\leq k$, if $(L_j ^{(n)}:L_{k+1}^{(n)} )$  is not contained in $(I^{n+1}:L_{k+1} ^ {(n)})$ and $L_j ^{(n)}$ belongs to an edge that comes before the edge $L_{k+1}^{(n)}$ belongs to, then  there exists $i \leq k$, such that $(L_i ^{(n)}: L_{k+1} ^{(n)})$ is generated by a variable, $(L_j ^{(n)} : L_{k+1} ^{(n)}) \subseteq (L_i ^{(n)} : L_{k+1} ^ {(n)} )$ and $L_i^{(n)}$ belongs to an edge that comes before or equal to the edge $L_j^{(n)}$ belongs to.
\\\\
\end{lemma}

\begin{proof}
  We prove the Lemma by induction on $n$. We recall that for two monomials $m_1$ and $m_2$, $(m_1:m_2)=(\frac{m_1}{\text{gcd}(m_1,m_2)})$. This is going to be used in several places.\\  

 If $n=1$,  $(L_j : L_{k+1})$ is either $(L_j)$, in which case $(L_j : L_{k+1} ) \subseteq (I^2 : L_{k+1} )$ or it is generated by a variable in which case we take $L_i = L_j$. Hence the lemma is true for $n=1$.\\
 
 Suppose the result is true for $n-1$. Let $L_j ^{(n)}$ belong to $ab$, so that $L_j ^{(n)}= ab M_1$ where $M_1 \in L^{(n-1)}$. By assumption $L_{k+1}^{(n)}$ belongs to an edge which comes after $ab$ in $L$. If neither $a$ nor $b$ divide $L_{k+1} ^ {(n)}$ then $(L_j ^{(n)} : L_{k+1} ^ {(n)} ) \subseteq (ab) \subseteq (I^{n+1} : L_{k+1} ^ {(n)})$ which is contrary to our assumption.\\
 
   Without loss of generality we assume $a | L_{k+1} ^{(n)}$. As $L_{k+1}^{(n)}$ is a product of edges,there exists an edge $ac$ with $ac |^{\text{edge}} L_{k+1}$, where $ac$ is a part of the maximal expression of $L^{(n)}_{k+1}$. So, $L_{k+1}^{(n)}= ac M_2$ for some $M_2 \in L^{(n-1)}$ which is the remaining part of the maximal expression. Now $ab \nmid ^{\text{edge}} L_{k+1} ^{(n)}$ as $L_{k+1}^{(n)}$ belongs to an edge that comes after $ab$. Hence $b\neq c$.\\
 
 If $(L_j ^{(n)}: L_{k+1} ^{(n)} ) \subseteq (b)$, then we take $L_i ^{(n)} = ab M_2$. Clearly $L_ i ^{(n)}$ belongs to $ab$ or some edge that comes before $ab$. Also, $(L_i ^{(n)} : L_{k+1} ^{(n)}) = (ab M_2 : ac M_2) = (b)$. Hence $L_i^{(n)}$ has all the required properties.\\
 
 If $(L_j ^{(n)}: L_{k+1} ^{(n)} )$ is not contained in $(b)$, then there is a variable $d$ such that $bd$ is an edge and $bd |^{\text{edge}} M_2$ and $bd$ is a part of maximal expression of $M_2$. Let $(L_j ^{(n)}: L_{k+1} ^{(n)} ) \subseteq (f)$ where $f$ is a variable. If $(L_j ^{(n)}: L_{k+1} ^{(n)} )=(f)$ then we take $L_i ^{(n)} = L_j ^{(n)}$. This has all the required properties.\\
 
 So let us assume $(L_j ^{(n)}: L_{k+1} ^{(n)} ) = (M_1 b : M_2 c) \subsetneq (f)$. Let $(L_j ^{(n)}: L_{k+1} ^{(n)} )= (fm)$ where $m$ is a monomial which is not $1$. So there is an edge $fg$ such that $fg | ^{\text{edge}} M_1$ and $fg$ is part of the maximal expression of $M_1$. If $g \nmid M_2 c$ then  $(L_j ^{(n)}: L_{k+1} ^{(n)} ) \subseteq (fg) \subseteq (I^{n+1} : L_{k+1} ^{(n)})$ which contradicts our assumption. So $g| M_2 c$.\\
 
 If $g=c$ then either $f=d$, that is $fcab=bdac$ or $(fcab:bdac)=(f)$. In the first case  $L_{k+1}= ac M_2= ac bd \frac{M_2}{bd} = fcab \frac{M_2}{bd}$. Now $bd |^{\text{edge}} M_2$, so $ab | ^{\text{edge}} L_{k+1} ^{(n)}$ which is a contradiction. In the second case we take $L_i ^{(n)} = (fc)(ab) \frac{L_{k+1} ^{(n)}}{bdac}$. Clearly $L_i ^{(n)} $ belongs to $ab$ or a some edge that comes before $ab$ and $(L_i ^{(n)}: L_{k+1} ^{(n)} )=(f)$, which contains $(L_j ^{(n)}: L_{k+1} ^{(n)} )$. Hence $L_i^{(n)}$ has the required properties.\\
 
 Now let us assume $g \neq c$. So there is an edge $gh$ such that $gh |^{\text{edge}} M_2$, such that $gh$ is a part of the maximal expression of $M_2$. Let $\frac{M_1}{fg} = N_1$ and $\frac{M_2}{gh}=N_2$. As $(L_j ^{(n)}:L_{k+1}^{(n)})=(fm)$, $fg ab N_1|fm gh ac N_2$. So $ab N_1| hm ac N_2$. So $(hm) \subset (ab N_1: ac N_2)$. We observe that $(ab N_1 : ac N_2)$ is either $(m)$ or $(hm)$. For if $m'|m$ then $abN_1| hm' ac N_2$ implies $fg ab N_1 | fm' gh ac N_2$ implies $fm|fm'$ implies $m=m'$.  \\\\
 
 If $(N_1 ab : N_2 ac )=(m)$  then $(L_j ^{(n)}: L_{k+1} ^{(n)} ) \subseteq (m)=(ab N_1 : ac N_2)$. Now both $ab N_1$ and $acN_2$ are in $L^{(n-1)}$. As $ab N_1$ belongs to $ab$ and $ac N_2$ belongs to some edge which comes after $ab$, $ab N_1 >_{\text{lex}} ac N_2$. By induction either $(ab N_1 : ac N_2) \subseteq (I^n  : ac N_2)$ or there exists $M_0$ in $L^{(n-1)}$, $M_0 >_{\text{lex}} ac N_2$, $(ab N_1 : ac N_2) \subseteq (M_0 : ac N_2)$, $(M_0 : ac N_2)$ is generated by a variable and  $M_0$ belongs to an edge that comes before or equal to $ab$. In the first case $(L_j ^{(n)}: L_{k+1} ^{(n)} ) \subseteq (ab N_1 : ac N_2) \subseteq (I^n : ac N_2) \subset (I^{n+1} : gh ac N_2)= (I^{n+1} : L_{k+1} ^{(n)})$, which is a contradiction. In the second case write $L_i^{(n)} = gh M_0$. We know that $L_i^{(n)} >_{\text{lex}} L_{k+1}^{(n)}$ as $M_0$ belongs to an edge that comes before or equal to $ab$. Also $(L_i ^{(n)}: L_{k+1} ^{(n)} ) = (M_0 : ac N_2)$, $(L_j ^{(n)}: L_{k+1} ^{(n)} ) \subseteq (m)=(ab N_1 : ac N_2) \subseteq (M_0 : ac N_2)$  and $(M_0 : ac N_2)$ is generated by a variable.\\
 
 Now let us assume $ (ab N_1 : ac N_2) = (hm)$. As $ab N_1 >_{\text{lex}} ac N_2$ , by induction either $(ab N_1 : ac N_2) \subseteq (I^n  : ac N_2)$ or there exists $M_0'$ in $L^{(n-1)}$, $M_0' >_{\text{lex}} ac N_2$, with $(ab N_1 : ac N_2) \subseteq (M_0' : ac N_2)$, $(M_0' : ac N_2)$ is generated by a variable, and $M_0'$ belongs to an edge that comes before or equal to $ab$. In the first case $hm ac N_2 \in I^n $, so $fm gh ac N_2= fg mh ac N_2 \in I^{n+1}$. So $(L_j ^{(n)}:L_{k+1}^{(n)}) \subseteq (I^{n+1}:L_{k+1}^{(n)})$, which is a contradiction. In the second case if $(M_0' : ac N_2) \neq (h)$ then let $L_i^{(n)} = gh M_0'$. As $M_0'$ belongs to an edge that comes before or equal to $ab$,  $L_i ^{(n)} >_{\text{lex}} L_{k+1}^{(n)}$.  Also $(L_i ^{(n)}: L_{k+1} ^{(n)} ) = (M_0' : ac N_2)$ which contains $(L_j ^{(n)}: L_{k+1} ^{(n)} )$ and is generated by a variable. If $(M_0' : ac N_2)=(h)$ we take $L_i^{(n)} =fg M_0'$. By same reasoning  $L_i ^{(n)} >_{\text{lex}} L_{k+1}^{(n)}$. As $L_i ^{(n)}$ can not be same as $L_{k+1}^{(n)}$ we observe $(L_i ^{(n)}: L_{k+1} ^{(n)})=(f)$. So this $L_i^{(n)}$ has all the required properties. This completes the proof.
\end{proof}

\begin{theorem}
For all $k\geq 1$ and for all $j\leq k$, if $(L_j ^{(n)}: L_{k+1} ^{(n)} )$  is not contained in $(I^{n+1}:L_{k+1} ^ {(n)})$ then  there exists $i \leq k$, such that $(L_i ^{(n)}: L_{k+1} ^{(n)})$ is generated by a variable and $(L_j ^{(n)} : L_{k+1} ^{(n)}) \subseteq (L_i ^{(n)} : L_{k+1} ^ {(n)} )$. \\\\
\end{theorem}

\begin{proof}
 We have $L_j ^{(n)}=mm_1$ and $L_{k+1}^{(n)}=mm_2$ where $m \in \text{Mingens} (I^k)$ and $m_1,m_2 \in \text{Mingens} (I^{n-k})$ with $m_1$ belongs to an edge that comes strictly before the edge $m_2$ belongs. We observe $(L_j^{(n)}:L_{k+1}^{(n)})=(m_1:m_2)$ and $(I^{n-k+1}:m_2) \subseteq (I^{n+1}:mm_2)$. With these two observations the theorem follows from Lemma $4.11$.
This finishes the proof.
\end{proof}

\section{Bounding the regularity: The Framework}

In this section we create the framework from which we shall prove our bounds. The framework is created by repeated use of Lemma $2.11$. Let $I$ and $J$ be two homogeneous square free monomial ideals in $S$ generated in degrees $n_1$ and $n_2$ respectively. Assume $J \subset I$ and $n_2$ is strictly greater than $n_1$. If the unique set of minimal monomial generators of $I$ is $\{m_1,m_2,...,m_k\}$ then repeated use of Lemma $2.11$ gives us the following lemma:

\begin{lemma}
Let $A= \max \{\reg (J:m_1)+n_1\}$ $$B=\max \{\reg ((J,m_1,..,m_l):m_{l+1})+n_1|1\leq l \leq {k-1}\}$$ $$C=\reg(I)$$ Then $\reg{J} \leq \max \{A,B,C\}$.

\end{lemma}

\begin{proof}

We consider the follwing short exact sequence:
\[
0 \longrightarrow \frac{S}{(J : m_1)}(-n_1) \overset{.m_1} \longrightarrow \frac{S}{J} \longrightarrow \frac{S}{(J, m_1)} \longrightarrow 0
\]

This gives us $\reg(J) \leq \max \{\reg(J:m_1)+n_1=A, \reg(J,m_1) \} $. Let $J_l:=((J,m_1,...,m_{l-1}) : m_l)$ for all $l \geq 2$. For all $1 \leq l \leq k-1$ we can consider the exact sequence
\[
0 \longrightarrow \frac{S}{(J_{l+1})}(-n_1) \overset{.m_{l+1}} \longrightarrow \frac{S}{(J,m_1,...,m_l)} \longrightarrow \frac{S}{(J,m_1,...,m_{l+1})} \longrightarrow 0,
\]

This gives us $$\reg(J,m_1,...,m_l) \leq \max \{\reg(J_{l+1})+n_1, \reg(J,m_1,...,m_{l+1}) \} $$ from which $\reg(J)\leq \max\{A,B,C\}$ follows.
\end{proof}

 This lemma together with Theorem $4.12$ gives the next theorem which is the main result we use for finding bounds on regularity of higher powers of edge ideals.\\
 
\begin{theorem}
For any finite simple graph $G$ and any $s\geq 1$, let the set of minimal monomial generators of $I(G)^s$ be $\{m_1,....,m_k\}$, then $$\reg(I(G)^{s+1}) \leq \max \{ \reg (I(G)^{s+1} : m_l)+2s, 1\leq l \leq k, \reg ( I(G)^s)\}.$$ 
\end{theorem}

\begin{proof}
 Minimal monomial generators of $I(G)^s$ forms the ordered list $L^{(s)}$ from section $4$. So by Lemma $5.1$, $$\reg (I(G)^{s+1}) \leq \max \{A,B,C\}$$ Where $$A=\max \{\reg (I(G)^{s+1}:L^{(s)}_1)+2s\}$$ $$B=\max \{\reg (((I(G)^{s+1},L^{(s)}_1,..,L^{(s)}_l):L^{(s)}_{l+1})+2s|1\leq l \leq {k-1}\}$$ $$C= \reg(I(G)^s).$$
 
 But in light of Theorem $4.12$, $((I(G)^{s+1},L^{(s)}_1,..,L^{(s)}_l):L^{(s)}_{l+1})$ is the same as \\$((I(G)^{s+1} : L^{(s)}_{l+1}), \text{some variables})$. So by Lemma $2.10$ $$\reg((I(G)^{s+1},L^{(s)}_1,..,L^{(s)}_l):L^{(s)}_{l+1}) \leq  \reg((I(G)^{s+1} : L^{(s)}_{l+1}),$$
 
  and the theorem follows.
\end{proof}

  As a corollary to the above theorem we get the following important result:
\begin{corollary}
If for all $s\geq 1$ and for all minimal monomial generator $m$ of $I(G)^s$, $\reg(I(G)^{s+1}:m) \leq 2$ and $\reg(I(G)) \leq 4$ then for all $s \geq 1, \reg(I(G)^{s+1})=2s+2$; as a consequence $I(G)^{s+1}$ has a linear minimal free resolution. 
\end{corollary}

\begin{proof}
We observe that under the condition if $\reg(I(G)^s) \leq 2s+2$ then $\reg (I(G)^{s+1}) \leq 2s+2$ too. Now $\reg(I(G)) \leq 4$ implies $\reg (I(G)^2 ) \leq 4$. By induction assume $\reg I(G)^k \leq 2k$. As $2k < 2k+2$, $\reg I(G)^k \leq 2k+2$. Hence $\reg I(G)^{k+1} \leq 2k+2$. This proves the corollary.
\end{proof}
\bigskip

\section{Bounding the regularity: The results}
\medskip

 In this section we give some new bounds on $\reg (I(G)^s)$ for certain classes of gap free graphs $G$. The main idea is to carefully analyze the ideal $(I(G)^{s+1}:e_1....e_s)$ for an arbitrary $s$-fold product of edges, i.e. for $i \neq j$, $e_i=e_j$ is a possibility. Now any $s$-fold product can be written as product of $s$ edges in various ways. In this section we fix a presentation and work with respect to that. We first prove that these ideals are generated in degree two for any graph $G$.\\
 
\begin{theorem}
For any graph $G$ and for any $s$-fold product $e_1....e_s$ of edges in $G$ (with the possibility of $e_i$ being same as $e_j$ as an edge for $i \neq j$), the ideal $(I(G)^{s+1} :e_1....e_s)$ is generated by monomials of degree two.
\end{theorem}

\begin{proof}
We prove this using induction on $s$. For $s=0$ the result is clear as $(I(G):(1))=I(G)$, which is generated by monomials of degree two. Now let us assume the theorem is true till $s-1$.\\

  Let $m$ be a minimal monomial generator of $(I(G)^{s+1}:e_1....e_s)$. Then $e_1....e_s m$ is divisible by an $s+1$-fold product of edges. By degree consideration $m$ can not have degree $1$. If $m$ has degree greater than or equal to $3$ then again by a degree consideration for some $i$, $e_i=pq$ such that $e_1...e_{i-1}qe_{i+1}..e_s m$ is divisible by an $s+1$-fold product of edges. Without loss of generality we may assume $e_1=pq$ and there is an $s+1$-fold product $f_1....f_{s+1}$ such that $f_1....f_{s+1}|qe_2....e_s m$.\\
  
  If $q|f_1.....f_{s+1}$, without loss of generality we may assume $f_1=p'q$. So \\ $p'qf_2....f_{s+1}|qe_2....e_s m$. Hence $f_2....f_{s+1} |e_2....e_s m$. If $q$ does not divide $f_1....f_{s+1}$ then $f_1....f_{s+1} | e_2....e_s m$ and hence $f_2....f_{s+1} | e_2....e_s m$. In both cases $m \in (I(G)^s : e_2....e_s)$.\\
  
   Now $(I(G)^s : e_2....e_s) \subset (I(G)^{s+1} : e_1....e_s)$ and $m$ is a minimal monomial generator of $(I(G)^{s+1}: e_1....e_s)$. So $m$ has to be a minimal monomial generator of $(I(G)^s: e_2....e_s)$. Hence by induction $m$ has degree two, which is a contradiction to the assumption that $m$ has degree greater than or equal to three. Hence $m$ has to have degree two.
\end{proof}

 To analyze the generators of $(I(G)^{s+1} :e_1....e_s)$, we introduce the notion of $\emph{even-connectedness}$ with respect to $s$-fold products.\\
 
\begin{definition}
Two vertices $u$ and $v$ ($u$ may be same as $v$) are said to be even-connected with respect to an $s$-fold product $e_1....e_s$ if there is a path $p_0 p_1....p_{2k+1}$, $k \geq 1$ in $G$ such that:\\
1. $p_0=u, p_{2k+1}=v.$\\
2. For all $0\leq l \leq k-1$, $p_{2l+1} p_{2l+2}=e_i$ for some $i$.\\
3. For all $i$, $$|\{l\geq 0| p_{2l+1} p_{2l+2} =e_i \} | \leq | \{j |e_j=e_i \} | $$
4. For all $0 \leq r \leq 2k$, $p_r p_{r+1}$ is an edge in $G$.\\

 If these properties are satisfied then $p_0,....,p_{2k+1}$ is said to be an even-connection between $u$ and $v$ with respect to $e_1....e_s$.
\end{definition}

\begin{eg}
Let $I(G)= (xy,xu,yv,yw,wz,zv)$ and $e_1=xy$, $e_2=wz$ then $u,x,y,w,z,v$ is an even-connection between $u$ and $v$ with respect to $e_1e_2$.
\end{eg}

 The following observation is an immediate consequence of the definition:\\
 
\begin{observation}
If $u=p_0,....,p_{2k+1}=v$ is an even-connection with respect to some $s$-fold product $e_1....e_s$, then for any $j'\geq j\geq 0$, any neighbor $x$ of $p_{2j+1}$ and any neighbor $y$ of $p_{2j'+2}$ are even connected with respect to $e_1....e_s$.\\
\end{observation}

The next theorem also easily follows from the definition.\\

\begin{theorem}
If $u=p_0,....,p_{2k+1}=v$ is an even-connection with respect to some $s$-fold product $e_1....e_s$ then $uv \in (I(G)^{s+1}:e_1....e_s)$.
\end{theorem}

\begin{proof}
By condition 2 and 3 of the definition, $e_1....e_s=p_1....p_{2k}.e_{j_1}....e_{j_{s-k}}$, for some $\{j_1,j_2,...,j_{s-k} \} \subset \{1,....,s\}$ and by condition 1 and 4 of definition $up_1....p_{2k} v$ is a $k+1$-fold product of edges in $G$. Hence $uve_1....e_s$ is an $s+1$-fold product of edges in $G$ and the result follows.
\end{proof}

 Although we fix a representation for all $s$-fold product and work with respect to that representation, it is worth noting that our definition of even-connectedness is independent of the representation we choose in the following sense:\\
 
\begin{theorem} If $f_1....f_s=e_1....e_s$ are two different representations of same $s$-fold product as product of edges and $u$ and $v$ are even-connected with respect to $e_1....e_s$, then $u$ and $v$ are even-connected with respect to $f_1....f_s$.
\end{theorem}

\begin{proof}
 Let $u=p_0,....,p_{2k+1}=v$ be an even-connection between $u$ and $v$ with respect to $e_1....e_s$. We shall construct an even-connection $q_0,....,q_{2r+1}$ between $u$ and $v$ with respect to $f_1....f_s$.\\
 
  Let $i$ be minimal such that $p_{2i+1}p_{2i+2}$ is not equal to any edge $f_1,...,f_s$. Let $q_0=p_0,...,q_{2i+1}=p_{2i+1}$. We have $(up_1)(p_2p_3)...(p_{2k} v) e_{t_1}....e_{t_{s-k}}=(uv)f_1....f_s$. Then $p_{2i+1}(p_{2i+2} p_{2i+3})....(p_{2k}v)e_{t_1}...e_{t_{s-k}}=vf_{j_1}....f_{j_{s-i}}$. If $v=p_{2i+1}$ we are done. Otherwise $p_{2i+1}$ divides one of the $f$s; without loss of generality let $f_{j_1}=p_{2i+1}q_{2i+2}$. If $vq_{2i+2}$ is an edge in $G$, we are done by taking $q_{2i+3}=v$. Otherwise we have $v q_{2i+2} f_{j_2}....f_{s-i}$ is an $(s-i)$-fold product of edges $g_1....g_{s-i}$, where without loss of generality $g_1=q_{2i+2} q_{2i+3}$ and $f_{j_2}=q_{2i+3} q_{2i+4}$. After selecting (without loss of generality) $g_l=q_{2i+2l} q_{2i+2l+1}$ and $f_{j_{l+1}}=q_{2i+2l+1} q_{2i+2l+2}$, we select $q_{2i+2l+3}$ inductively. If $v q_{2i+2l+2}$ is an edge in $G$, we are done by choosing $q_{2i+2l+3}=v$. Other wise, $g_{l+1}....g_{s-i}=vq_{2i+2l+2}f_{j_{l+2}}....f_{j_{s-i}}$.  If $v$ is  connected to $q_{2i+2l+2k}$ for some $k$ in $G$ then we are done by choosing $q_{2i+2l+2k+1}=v$. If not then $g_1....g_{s-i} = v g_1g_2...g_{s-i-1} q_{2i+2s-2} $; but this will force $g_{s-i}=q_{2i+2s-2} v$, contradicting the fact that $v$ is not connected to $q_{2i+2l+2k}$ for any $k$.\\
  
  The conditions $1,2,4$ of the definition are automatically satisfied by our construction. Condition $3$ is satisfied because each $q_{2i+1} q_{2i+2}$ is $f_{r_i}$ for some integer $r_i$ and $q_{2i+3} q_{2i+4}$ is some $f_{r_{i+1}}$ where $r_{i+1} \notin \{r_1,...,r_i\}$.
\end{proof}

 We now observe that all edges of $G$ belong to $(I(G)^{s+1} : e_1....e_s)$. If $uv$, $u$ may be equal to $v$, belongs to $(I(G)^{s+1} :e_1....e_s)$ and $uv$ is not an edge, then we prove that $u$ and $v$ has to be even-connected with respect to the $s$-fold product $e_1....e_s$. The conditions $1,2,3,4$ are satisfied by the way of construction.\\
 
\begin{theorem}
Every generator $uv$ ($u$ may be equal to $v$) of  $(I(G)^{s+1} : e_1....e_s)$ is either an edge of $G$  or even-connected with respect to $e_1....e_s$, for $s \geq 1$.
\end{theorem}

\begin{proof}
Suppose $uv$ is not an edge and $u$ and $v$ are not even-connected. Now $uve_1....e_s=f_0...f_s$ is an $s+1$-fold product of edges, where $f_0=up_0$ such that there is an edge $e_{i_0}=p_0 q_1, 1\leq i_{0} \leq s$. After selecting $f_j=q_j p_j$ and $e_{i_j}=p_j q_{j+1}$, $1 \leq i_j \leq s$ and all $i_j$ are different, we select $f_{j+1}$ and $e_{i_{j+1}}$ inductively. $q_{j+1}$ is part of an edge $q_{j+1} p_{j+1}$ in the $s+1$ fold product $f_0...f_s$. We choose $f_{j+1}=q_{j+1} p_{j+1}$. Now as $u$ and $v$ are not even-connected $p_{j+1}$ is not $v$. So it is part of an edge amongst the remaining $e_i$ s. So there exists $e_{i_{j+1}}=p_{j+1} q_{j+2}$, $ i_{j+1} \in \{1,..,s\} \setminus \{i_1...i_j\}$. Now as $u$ and $v$ are not even-connected, $v \neq p_k$ for any $k$. We observe $f_0..f_s = u (p_0 q_1) (p_1 q_2)...(p_{s-1} q_s) p_s =uv e_1...e_s$. By construction $(p_0 q_1) (p_1 q_2)..(p_{s-1} q_s)= e_1...e_s$. This forces $p_s=v$, which is a contradiction.
\end{proof}

\begin{eg}
Let $I(G)=(xy,xu,xv,xz,yz,yw)$. Then $(I(G)^2:xy)=I(G)+(z^2,uz,vz,wz,uw,vw)$. Here $z$ is even-connected to itself and $u,v,w$ with respect to $xy$; also $u,w$ and $v,w$ are even-connected with respect to $xy$.
\end{eg}

  We observe that $(I(G)^{s+1}:e_1....e_s)$ need not be square free as there is a possibility that some vertex $u$ is even-connected to itself with respect to $e_1....e_s$. So we polarize $(I(G)^{s+1}:e_1....e_s)$ to get a square free quadratic monomial ideal (i.e. an edge ideal) $(I(G)^{s+1} : e_1....e_s)^{\p}$. For details of polarization we refer to [9], section $3.2$ of [10] and exercise $3.15$ of [10]. Here we just recall the definition and one theorem which states a quadratic monomial ideal and its polarization have same regularity.
  
\begin{definition}
For any quadratic monomial ideal $I$ in $K[x_1....x_n]$, $I^{\p}$ is a square free quadratic monomial ideal in $K[x_1,....,x_n,x_1',....,x_n']$ where $I^{\p}=<x_i x_j,x_k x_k'| x_i x_j \in I ,x_k^2 \in I>$. 
\end{definition}

 The following theorem, which we state without proof is a special case of Proposition $1.3.4$ of [9], we also refer to section $3.2$ and exercise $3.15$ of [10].\\
 
\begin{theorem}
$\reg(I^{\p})=\reg(I)$.
\end{theorem}

   Clearly by Theorems $6.1$, $6.5$ and $6.7$, $(I(G)^{s+1} : e_1....e_s)^{\p}$ is an edge ideal with the same regularity as $\reg (I(G)^{s+1} : e_1....e_s)$. We describe the graph associated to this edge ideal in the following Lemma:\\
 
\begin{lemma}
$(I(G)^{s+1} : e_1....e_s) ^{\p}$ is the edge ideal of a new graph $G'$ which has:\\
1. All vertices and edges of $G$.\\
2. Any two vertices $u,v, u\neq v$ of $G$ that are even-connected with respect to $e_1....e_s$ are connected by an edge in $G'$.\\
3. For every vertex $u$ which is even connected to itself with respect to $e_1....e_s$, there is a new vertex $u'$ which is connected to $u$ by an edge and not connected to any other vertex (so $uu'$ is a whisker).
\end{lemma}

\begin{proof}
By Theorem $6.7$, every generator $uv$ ($u$ may be equal to $v$) of  $(I(G)^{s+1} : e_1....e_s)$ is either an edge of $G$  or even-connected with respect to $e_1....e_s$, for $s \geq 1$. If it is an edge in $G$, it satisfies condition 1; if it is an even-connection with $u \neq v$ it satisfies condition 2; if it is an even-connection with $u=v$, then by definition of polarization there will be a whisker $u'$ on $u$ in $G'$ and hence it will satisfy condition 3. Conversely edges described by the conditions 1,2 and 3 belong to $G'$ by Theorems $6.5$ and $6.7$.
\end{proof}

\begin{eg}
Let $G$ be the following graph:

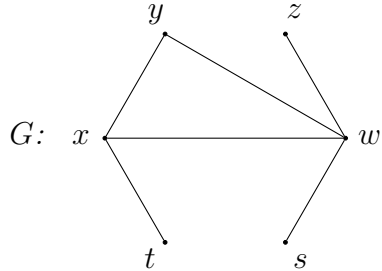
\begin{center}
\begin{tikzpicture}
[scale=.8, vertices/.style={draw, fill=black, circle, inner sep=0.5pt}]
\useasboundingbox (-2.5,-2.4) rectangle (2.5,2.4);
\node [anchor=base] at (-3.3,-.1){$G$:};
\node [vertices] (1) at (0:2) {};
\node [anchor=base] at (2.4,-.1) {$w$};
\node [vertices] (2) at (60:2) {};
\node [anchor=base] at (60:2.3) {$z$};
\node [vertices] (3) at (120:2) {};
\node [anchor=base] at (120:2.3) {$y$};
\node [vertices] (4) at (180:2) {};
\node [anchor=base] at (-2.4, -.1) {$x$};
\node [vertices] (5) at (240:2) {};
\node [anchor=base] at (240:2.5) {$t$};
\node [vertices] (6) at (300:2) {};
\node [anchor=base] at (300:2.5) {$s$};
\foreach \to/\from in {2/1, 3/1, 4/1, 4/3, 5/4, 6/1}
	\draw [-] (\to)--(\from);
\end{tikzpicture}
\end{center}

Then the graph $G'$ associated to $(I(G)^2:xw)^\p$ is the following:

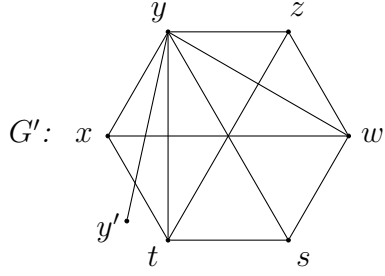
\begin{center}
\begin{tikzpicture}
[scale=.8, vertices/.style={draw, fill=black, circle, inner sep=0.5pt}]
\useasboundingbox (-2.5,-2.4) rectangle (2.5,2.4);
\node [anchor=base] at (-3.3,-.1){$G'$:};
\node [vertices] (1) at (0:2) {};
\node [anchor=base] at (2.4,-.1) {$w$};
\node [vertices] (2) at (60:2) {};
\node [anchor=base] at (60:2.3) {$z$};
\node [vertices] (3) at (120:2) {};
\node [anchor=base] at (120:2.3) {$y$};
\node [vertices] (4) at (180:2) {};
\node [anchor=base] at (-2.4, -.1) {$x$};
\node [vertices] (5) at (240:2) {};
\node [anchor=base] at (240:2.5) {$t$};
\node [vertices] (6) at (300:2) {};
\node [anchor=base] at (300:2.5) {$s$};
\node [vertices] (7) at (220:2.2) {};
\node [anchor=base] at (220:2.6) {$y'$};
\foreach \to/\from in {2/1, 3/1, 3/2, 4/1, 4/3, 5/2, 5/3, 5/4, 6/1, 6/3, 6/5,7/3}
	\draw [-] (\to)--(\from);
\end{tikzpicture}
\end{center}

\end{eg}

 Next we prove several lemmas that will be useful to get our main results.\\
 
\begin{lemma}
Suppose $u=p_0,....,p_{2k+1}=v$ is an even-connection between $u$ and $v$ and $z=q_0,....,q_{2l+1}=w$ is an even connection between $z$ and $w$, both with respect to $e_1....e_s$. If for some $i$ and $j$, $p_{2i+1} p_{2i+2}$ and $q_{2j+1} q_{2j+2}$ has a common vertex in $G$ then $u$ is even-connected to either $z$ or $w$ with respect to $e_1....e_s$  and $v$ is even-connected to either $z$ or $w$ with respect to $e_1....e_s$.
\end{lemma}

\begin{proof}
We prove it for $u$, and the proof for $v$ follows by symmetry. Let $i$ be the smallest integer such that there is $j$  with the required property. If $p_{2i+1}= q_{2j+1}$ then $u=p_0,...,p_{2i+1}=q_{2j+1},q_{2j+2},q_{2j+3}
,...,q_{2l+1}=w$ gives an even-connection between $u$ and $w$ with respect to $e_1....e_s$ (conditions 1,2 and 4 are automatically satisfied and condition 3 is satisfied as $i$ is the smallest integer such that there is a $j$). Similary if $p_{2i+1}= q_{2j+2}$ then $u=p_0,...,p_{2i+1}=q_{2j+2},q_{2j+1},q_{2j}
,...,q_{0}=z$ gives an even-connection between $u$ and $z$ with respect to $e_1....e_s$; if $p_{2i+1}$ is not same as either $q_{2j+1}$ or $q_{2j+2}$ and $p_{2j+2}=q_{2j+1}$ then $u=p_0,...,p_{2i+1},p_{2j+2}=q_{2j+1},q_{2j+2},q_{2j+1},q_{2j}
,...,q_{0}=z$ gives an even-connection between $u$ and $z$ with respect to $e_1....e_s$; if $p_{2i+1}$ is not same as either $q_{2j+1}$ or $q_{2j+2}$ and $p_{2j+2}=q_{2j+2}$ then $u=p_0,...,p_{2i+1},p_{2j+2}=q_{2j+2},q_{2j+1},q_{2j+2},
,...,q_{2l+1}=w$ gives an even-connection between $u$ and $w$ with respect to $e_1....e_s$; in each of these cases conditions 1,2 and 4 are satisfied automatically and condition 3 is satisfied as $i$ is the smallest integer with the property. This covers all the cases.
\end{proof}

 The next two lemmas are results about gap free graphs:\\
 
\begin{lemma}
If $G$ is gap free then so is the graph $G'$ associated to $(I(G)^{s+1}: e_1....e_s)^{\p}$, for every $s$-fold product $e_1....e_s$.
\end{lemma}

\begin{proof}
There are three possibilities of gap formation in $G':$\\ 1. Between two edges from $G$.\\
2. Between two edges that are not edges in $G$.\\
3. Between two edges where one of them is an edge in $G$ another is not.\\

  No two edges in $G$ can form a gap in $G$ as $G$ is gap free. So they can't form an edge in $G'$ as in $G'$ no edge of $G$ is being deleted.\\
  
  For the second case suppose $uv$ and $zw$ are even-connected with respect to $e_1....e_s$ and neither $uv$ nor $zw$ is an edge in $G$. Without loss of generality we may assume $\text{gcd} (uv,zw)=1$ as there is no question of gap formation otherwise. Let $u=p_0,....,p_{2k+1}=v$ be an even-connection between $u,v$ with respect to $e_1....e_s$ and let $z=q_0,....,q_{2l+1}=w$ be an even-connection between $z,w$ with respect to $e_1....e_s$. In light of Lemma $6.13$, we may assume for no $i,j$, $p_i = q_j$. If $u=q_1$ then $zu=zq_1$ is an edge in $G$ and if $z=p_1$ then $uz=up_1$ is an edge in $G$, so there is nothing to prove. Otherwise as $up_1$ and $zq_1$ are edges in $G$ and $G$ is gap free there are four possibilities:\\  
a. $u$ is connected to $z$ in $G$, in which case $uv$ (or $uu'$ in case $u=v$) and $zw$ (or $zz'$ in case $z=w$) can't form a gap, as in that case $uz$ is an edge in $G'$ too.\\
b. $p_1$ is connected to $z$, in which case $z,p_1,....,p_{2k+1}=v$ is an even-connection between $z$ and $v$ in $G$ so $zv$ is an edge in $G'$ hence $uv$ (or $uu'$ if $u=v$) and $zw$ (or $zz'$ if $z=w$) can't form a gap.\\
c. $p_1$ is connected to $q_1$, in which case $v=p_{2k+1},p_{2k},....,p_1,q_1,q_2,.....,q_{2l+1}=w$ gives an even-connection between $v$ and $w$, and $vw$ is an edge in $G'$.\\
d. $q_1$ is connected to $u$, in which case $u,q_1,....,q_{2l+1}=w$ is an even-connection between $u$ and $w$ in $G$ so $uw$ is an edge in $G'$ hence $uv$ (or $uu'$ if $u=v$) and $zw$ (or $zz'$ if $z=w$) can't form a gap.\\

  In the third case, $u,v$  are even-connected with respect to $e_1....e_s$ and $zw$ is an edge in $G$ and $uv$ is not an edge in $G$. Like before, we may assume $\text{gcd} (uv,zw)=1$. Let $u=p_0,....,p_{2k+1}=v$ be an even-connection between $u,v$ with respect to $e_1....e_s$. If $z=p_1$ then $uz=up_1$ is an edge in $G$ and if $w=p_1$ then $uw=up_1$ is an edge in $G$, so there is nothing to prove in these cases. Otherwise as $up_1$ and $zw$ are edges in $G$ and $G$ is gap free there are four choices:\\  
a. $u$ is connected to $z$, in which case $uv$ (or $uu'$ in case $u=v$) and $zw$ can't form a gap as in that case $uz$ is an edge $G'$ too.\\
b. $p_1$ is connected to $z$, in which case $z,p_1,....,p_{2k+1}=v$ is an even-connection between $z$ and $v$ in $G$ so $zv$ is an edge in $G'$ hence $uv$ (or $uu'$ if $u=v$) and $zw$  can't form a gap.\\
c. $p_1$ is connected to $w$, in which case $v=p_{2k+1},p_{2k},....,p_1,w$ is an even-connection; hence $uv$ and $zw$ can not form a gap.\\
d. $w$ is connected to $u$, in which case $uw$ is an edge in $G$, hence in $G'$.\\

  This finishes the proof.
\end{proof}

\begin{lemma}
Suppose $G$ is gap free. If $w_1,....,w_n$ is an anticycle in the graph $G'$ defined by $(I(G)^{s+1}:e_1....e_s)$ for some $s\geq 1$ and for $n \geq 5$, then $w_1,....,w_n$ is an anticycle in $G$.
\end{lemma}

\begin{proof}
 First of all, whiskers on any vertex can not be part of any anticycle of length $\geq 5$ as they only have degree $1$. Observe that it is enough to prove that for all $i,j$, $w_i, w_{i+j}$ are never even-connected with respect to $e_1....e_s$. Suppose on the contrary such $i,j$ exists. Without loss of generality we may choose $j$ to be minimal such that for some $i$, $w_i$ and $w_{i+j}$ are even-connected with respect to $e_1....e_s$. Observe that $j\geq 2$ as $w_i w_{i+1}$ can't be connected in an anticycle. Without loss of generality we may further assume $w_1$ and $w_{1+j}$ are even-connected with respect to $e_1....e_s$ via $w_1=p_0,p_1,....,p_{2k+1}=w_{1+j}$. Now observe $w_{2+j}$ is not connected to $p_1$ by an edge in $G$ as that will force $w_{1+j}$ and $w_{2+j}$ to be connected in $G'$ by observation $6.4$ leading to a contradiction. So there exists a smallest $l \geq 0$, $2+j \leq n-l \leq n$ such that $w_{n-l}$ is not connected to $p_1$ by an edge  in $G$. If $l=0$, then $w_{n}$ is not connected to $p_1$ by an edge in $G$ and if $l>0$ then $w_{n-l}$ is not connected to $p_1$ by an edge to $p_1$  in $G$ and $w_n, w_{n-1},..,w_{n-l+1}$ are connected to $p_1$ by an edge in $G$\\
 
   Next, we look at the edge $w_2 w_{n-l}$ in $G'$. If $w_2$ is connected to $p_1$ in $G$ then $w_2,p_1,...,p_{2k+1}=w_{1+j}$ will be an even connection that will violate the minimality of $j$. If $w_2$ is connected to $p_2$ in $G$ then by Observation $6.4$ $w_1 w_2$ has to be an edge in $G'$, which will contradict the fact $w_1....w_n$ is an anticycle. We observe $w_{n-l}$ can't be connected to $p_1$ by selection. If $w_{n-l}$ is connected to $p_2$ and $l=0$  then by Observation $6.4$ $w_1$ and $w_n$ have to be connected to each other in $G'$. If $w_{n-l}$ is connected to $p_2$ and $l>0$ then by Observation $6.4$ $w_{n-l+1}$ and $w_{n-l}$ have to be connected to each other in $G'$. Both cases lead to a contradiction as $w_1....w_n$ is an anticycle, so $w_2$ and $w_{n-l}$ are not connected to each other in $G$ and neither of them are connected to $p_1$ or $p_2$ (and hence $w_2,w_{n-l},p_1,p_2$ are four distinct vertices). As $p_1p_2$ is an edge in $G$, $w_2w_{n-l}$ can not be an edge in $G$; otherwise they will form a gap. So $w_2$ and $w_{n-l}$  are even-connected with respect to $e_1....e_s$. Let $w_2=q_0,....,q_{2r+1}=w_{n-l}$ be an even connection between $w_2$ and $w_{n-l}$ with respect to $e_1....e_s$.\\
   
   If for some $t_1,t_2 \geq 0$, $p_{2t_1+1} p_{2t_1+2}$ and $q_{2t_2+1} q_{2t_2+2}$ are the same edges of $G$ then by Lemma $6.13$, $w_2$ has to be even connected to either $w_1$ or $w_{1+j}$. The first case is not possible as $w_1..w_n$ is an anticycle and the second case is not possible by the minimality of $j$. So for no $t_1,t_2 \geq 0$, $p_{2t_1+1} p_{2t_1+2}$ and $q_{2t_2+1} q_{2t_2+2}$ are the same edges of $G$. So  we look at $w_{n-l} q_{2r}$ and $p_1 p_2$. Observe that $p_1$ is not connected to $w_{n-l}$ because of the selection. If $w_{n-l}$ is connected to $p_2$ and $l=0$  then by Observation $6.4$ $w_1$ and $w_n$ have to be connected to each other in $G'$. If $w_{n-l}$ is connected to $p_2$ and $l>0$ then by Observation $6.4$ $w_{n-l+1}$ and $w_{n-l}$ have to be connected to each other in $G'$. Both cases lead to a contradiction as $w_1....w_n$ is an anticycle. So $p_2$ is not connected to $w_{n-l}$ in $G$. If $p_1$ is connected to $ q_{2r}$ then $w_2$ and $w_{1+j}$ will be even-connected with respect to $e_1....e_s$ violating the minimality of $j$. If $p_2$ is connected to $q_{2r}$ then $w_1$ and $w_2$ will be even-connected and hence connected in $G'$.\\
   
   Hence for no $i,j$ are $w_i$ and $w_{i+j}$ even-connected with respect to $e_1....e_s$. So $w_1....w_n$ is an anticycle in $G$.
\end{proof}

 Using this lemma we get the following theorem of Herzog, Hibi and Zheng (Theorem $1.2$ of [13]) as a corollary:
 
\begin{theorem}
If $I(G)$ has linear resolution, then for all $s \geq 2$,  $I(G)^s$ has regularity $2s$. In other words $I(G)^s$ has a linear minimal free resolution.
\end{theorem}

\begin{proof}
As $I(G)$ has a linear resolution, it is gap free and hence the polarizations of all $(I(G)^{s+1}:e_1....e_s)$ are gap free by Lemma $6.14$, and any anticycle of length $\geq 5$ in the polarization of $(I(G)^{s+1}: e_1....e_s)$ is an anticycle of $G$ by Lemma $6.15$. But as $I(G)$ has linear resolution $G$ does not have an any anticycle.  By Theorem $2.12$   $\reg(I(G)^{s+1}:e_1....e_s)^\text{pol}=2$ for all $e_1....e_s$. Hence $\reg(I(G)^{s+1})=2s+2$ by Theorem $5.2$ and Theorem $6.10$.
\end{proof}

 Next we prove that for any gap free and cricket free graph $G$, and for all $s \geq 2$, $\reg (I(G)^s) = 2s$. This result is our main new result in this paper. This answers Question 1.1 partially. This also generalizes Nevo's result (Theorem $1.2$ of [12]) that for any gap free and claw free graph $G$, $\reg I(G)^2 =4$.\\
 
\begin{theorem}
For any gap free and cricket free graph $G$ and for all $s \geq 2$, $\reg (I(G)^s) = 2s$.
\end{theorem}

\begin{proof}
 In light of Theorem $2.12$, Theorem $3.4$, Corollary $5.3$, Theorem $6.10$ and Lemma $6.14$, it is enough to show the polarization of $(I(G)^{s+1}:e_1....e_s)$ does not have any anticycle $w_1....w_n$ for $n \geq 5, s \geq 1$, for every $s$-fold product $e_1....e_s$.\\ 
 
  Suppose $w_1...w_n$, $n\geq 5$, is an anticycle in the polarization of $(I^{s+1}:e_1...e_s)$ and $e_1=xy$. By Lemma $6.15$ $w_1....w_n$ is also an anticycle of $G$. Either $w_1$ or $w_3$ is a neighbor of $x$ or neighbor of $y$ else $w_1 w_3$ and $e_1$ forms a gap in $G$, a contradiction. Without loss of generality, we may assume $w_1$ is a neighbor of $x$. Now neither $w_2$ nor $w_n$ can be $x$ as they are not connected to $w_1$; also neither of them are $y$ as if
say $y=w_2$ then $w_nxyw_1$ is an even connection hence $w_1w_n$ is an edge in $G'$, a contradiction to the assumption on anticycle; similar thing happens if $y=w_n$. By Observation $6.4$ every neighbor of $y$ is connected to every neighbor of $x$ in $G'$. As neither $w_1 w_n$, nor $w_1 w_2$ is an edge in $G'$, neither of $w_2$ and $w_n$ are neighbors of $y$ in $G$. So one of them has to be neighbor of $x$, as $G$ is gap free. Again, without loss of generality, we may assume $w_2$ is a neighbor of $x$. Next we consider $w_3 w_n$. As $w_1$ and $w_2$ are neighbors of $x$ and neither $w_1w_n$ nor $w_2w_3$ are edges in $G'$, by Observation $6.4$ neither $w_3$ nor $w_n$ can be neighbor of $y$. Neither $w_3$ nor $w_n$ can be $x$ as they are $w_2w_3$ and $w_1w_n$ are not edges in $G'$. If $w_3=y$, as $w_1 w_3$ is an edge in $G$, $w_1$, being a neighbor of $y$, has to be connected to $w_2$, which is a neighbor of $x$ in $G'$ by Observation $6.4$. That will force $w_1 w_2$ to be an edge in $G'$, which is a contradiction. Similarly if $w_n=y$, $w_3$ being a neighbor of $y$ has to be connected to $w_2$ in $G'$ leading to a contradiction. Then either $w_3$ or $w_n$ has to be a neighbor of $x$. Without loss of generality we may assume $w_3$ is a neighbor of $x$. Notice that $y$ is not connected to $w_1$ in $G$ as that will force $w_2$, a neighbor of $x$ to be connected to $w_1$ in $G'$ leading to a contradiction. Hence $\{y,w_2,x,w_1,w_3 \}$ forms a cricket.
\end{proof}

  Next we prove that for any gap free graph $G$ with $\reg(I(G))=r$, the $\reg(I(G)^s)$ is bounded above by $2s+r-1$. But to do that we need a lemma about $``$longest" connections. Observe that if $G'$ is the graph associated to the polarization of $(I(G)^{s+1}: e_1...e_s)$, for some $s$-fold product, and $u$,$v$ are even-connected with respect to $u=p_0,....,p_{2k+1}=v$, then $uv$ is not only an edge in $G'$ but also an edge in the graph $(G'-\{y_1,...y_l\})$ for any set of points $y_1,....,y_l$ as long as $u,v \notin \{y_1,....,y_l\}$. We further emphasize that some of the $p_i$s can also belong to $\{y_1,....,y_l\}$ as long as they are not same as $u$ or $v$. \\

\begin{lemma}
Let $G'$ be the graph associated to the polarization of $(I(G)^{s+1}: e_1...e_s)$ for some $s$-fold product. Let us assume $u$,$v$ are even-connected with respect to $u=p_0,....,p_{2k+1}=v$. Suppose for some set of vertices $\{y_1,....,y_l\}$ we have $u,v \notin \{y_1,....,y_l\}$. Let us also assume for any other even-connection $u'=p'_0,....,p'_{2k'+1}=v'$ such that $u',v' \notin \{y_1,....,y_l\}$ we have $k' \leq k$. Then $(G'-\{y_1,....,y_l\} -\st u)$ is $G'' \cup \text{\{isolated whisker vertices\}}$, where $G''$ is a subgraph of $G$ obtained by deleting vertices.
\end{lemma}
  
\begin{proof}
For the set of points $\{y_1,....,y_l\}$, $uv$ is an edge in $(G'-\{y_1,...,y_l\})$ such that $u,v \notin \{y_1,....,y_l\}$ are even-connected with respect to $e_1....e_s$ via $u=p_0,p_1,p_2,..\\..,p_{2k+1}=v$. We also have that $k$ is maximum over all such even-connected edges in $(G'-\{y_1,...,y_l\})$. Let $u'v'$ be any edge in $(G'-\{y_1,...,y_l\})$ such that $u',v' \notin \{y_1,....,y_l\}$ and they are even-connected with respect to $e_1....e_s$ via $u'=x_0,x_1,x_2,....,x_{2k'+1}=v'$. If for any $j,j'$, $p_{2j+1} p_{2j+2}$ and $x_{2j'+ 1} x_{2j'+2}$ form the same edge in $G$ then by Lemma $6.13$, either $u'$ or $v'$ will be not a vertex in $(G'-\{y_1,....,y_l\}-\st u)$. Now observe, if for any $j,j'$, $p_{2j+1} p_{2j+2}$ and $x_{2j'+ 1} x_{2j'+2}$ do not form same edge in $G$ then either $x_1$ or $x_2$ has to be connected to $p_1$ or $p_2$ to avoid $x_1 x_2$ and $p_1 p_2$ forming a gap. If any of them (for example $x_1$) is connected to $p_1$ in $G$ that will make $\{v'=x_{2k'+1},x_{2k'},...,x_1,p_1,....,p_{2k+1}\}$ a longer connection violating the maximality of $k$. A similar thing happens if $x_2$ is connected to $p_1$ in $G$. So either of them has to be connected to $p_2$. If $x_1$ is connected to $p_2$ in $G$ then $u$ is connected to $v'$ in $G'$ as $u,p_1,p_2,x_1,...,x_{2k'+1}=v'$ will be an even-connection. Similarly if $x_2$ is connected to $p_2$ then $u$ is connected to $u'$ in $G'$ as $u,p_1,p_2,x_2,x_1,u'$ will be an even-connection. In both the cases either $u'$ or $v'$ will not be a vertex in $(G'-\{y_1,....,y_l\}-\st u)$. This proves that any edge in $(G'-\{y_1,....,y_l\}-\st u)$ is an edge in $G$. Hence the Lemma follows.
\end{proof}

 Using Lemma $6.18$ we prove the next theorem which guarantees that the gap between the regularity of powers of edge ideals of gap free graphs and the regularity of monomial ideals generated in the same degree and having a linear resolution, can not be arbitrarily large:
  
\begin{theorem}
For any gap free graph $G$ with $\reg(I(G))=r$ and any $s \geq 2$ the $\reg(I(G)^s)$ is bounded above by $2s+r-1$.
\end{theorem}

\begin{proof}
 Let $G'$ be the graph associated to the polarization of $(I(G)^{s+1}: e_1....e_s)$. We have $\reg(G') \leq \max \{\reg(G'-\st x)+1, \reg(G'-x)\}$ by Lemma $3.2$ for each vertex $x$. We choose $u_1$ and $v_1$ even connected by $u_1=p_0,....,p_{2k_1+1}=v_1$ such that $k_1$ is maximum. By Lemma $6.18$ $(G'-\st u_1)$ is a subgraph of $G$ obtained by vertex deletion along with some isolated whisker vertices. As isolated vertices do not affect the regularity of edge ideal, $\reg((G'-\st u_1) \leq r$ by Lemma $2.10$.\\
 
 Next we apply Lemma $3.2$ on $(G'-u_1)$, from which we delete a vertex $u_2$ which is even-connected to another vertex $v_2$ via $u_2=q_0,....,q_{2k_2+1}=v_2$ with $k_2$ maximum. Again by Lemma $6.18$ $(G'-u_1-\st u_2)$ is a subgraph obtained from $G-u_1$ by deletion of vertices along with some whisker vertices. Hence $\reg (G'-u_1-\st u_2) \leq r$. We keep selecting $u_1,u_2,...$ and apply Lemmas $3.2$ and $6.18$. As we are in a finite setup, for
some $l, (G'-{u_1,...,u_l})$ itself is a subgraph of $G$  obtained by repeated vertex deletion along with some isolated whisker vertices and $reg(G') \leq r+1$. Therefore, by Theorem $5.2$ and induction the result follows. 
\end{proof}

\textbf{Acknowledgements.} The author is very grateful to his advisor C. Huneke for constant support, valuable ideas and suggestions throughout the project. The author thanks G. Caviglia, A. De Stefani, J. Martin, J. Mermin and T. Vu for valuable discussions. The author also thanks the reviewers for suggesting many improvements and one anonymous reviewer for suggesting a simpler proof of Observation $4.10$.

\end{document}